\documentclass[12pt]{article}
\usepackage{fullpage,amsthm,amsmath,amssymb,xspace,url,relsize}

\newcounter{thmctr}
\newtheorem{thm}[thmctr]{Theorem}
\newtheorem{prop}[thmctr]{Proposition}
\newtheorem{cor}[thmctr]{Corollary}
\newtheorem{lemma}[thmctr]{Lemma}
\newtheorem*{definition}{Definition}
\theoremstyle{definition}
\newtheorem*{constr}{Construction}
\newtheorem*{example}{Example}
\theoremstyle{plain}

\newcommand{\discp}[1]{\texttt{Disc}$^{(k)}(#1,\mathsmaller{\mathsmaller{\geq}}p,\mu)$\xspace}
\newcommand{\discIk}{\discp{\mathcal{I} }}
\newcommand{\disca}[1]{\texttt{Disc}$^{(k-1)}(#1,\mathsmaller{\mathsmaller{\geq}}\alpha,\mu)$\xspace}
\newcommand{\discImone}{\disca{\mathcal{J} }}

\newcommand{\discIstrong}{\texttt{Disc}$^{(k)}(\mathcal{I},p,\mu)$\xspace}

\newcommand{\cdk}{\discp{\binom{[k]}{k-1} }}
\newcommand{\cdkmone}{\disca{\binom{[k-1]}{k-2} }}

\newcommand{\discthree}{\texttt{Disc}$^{(3)}(\mathcal{I},\mathsmaller{\mathsmaller{\geq}}p,\mu)$\xspace}
\newcommand{\disctwo}{\texttt{Disc}$^{(2)}(\mathcal{I},\mathsmaller{\mathsmaller{\geq}}p,\mu)$\xspace}
\newcommand{\discone}{\texttt{Disc}$^{(1)}(\{\emptyset\},\mathsmaller{\mathsmaller{\geq}}\alpha,\mu)$\xspace}

\DeclareMathOperator{\inj}{\text{inj}}

\newcommand{\dhruvuni}{University of Illinois at Chicago \\ mubayi@math.uic.edu}
\newcommand{\johnuni}{University of Illinois at Chicago \\ lenz@math.uic.edu}
\newcommand{\dhruvfoot}{\footnote{Research supported in part by  NSF Grants 0969092 and 1300138.}}
\newcommand{\johnfoot}{\footnote{Research partly supported by NSA Grant H98230-13-1-0224.}}

\title{Perfect Packings in Quasirandom Hypergraphs II}
\author{John Lenz \johnfoot \\ \johnuni \and Dhruv Mubayi \dhruvfoot \\ \dhruvuni}

\begin{document}

\maketitle

\begin{abstract}
For each of the notions of hypergraph quasirandomness that have been studied, we identify a large
class of hypergraphs $F$ so that every quasirandom hypergraph $H$ admits a
perfect $F$-packing.  An informal statement of a special case of our general result for 3-uniform
hypergraphs is as follows.  Fix an integer $r \geq 4$ and $0<p<1$. Suppose that $H$ is an $n$-vertex
triple system with $r|n$ and the following two properties:
\begin{itemize}
  \item for every graph $G$ with $V(G)=V(H)$, at least $p$ proportion 
of the triangles in $G$ are also edges of $H$,

  \item for every vertex $x$ of $H$, the link graph of $x$ is a 
quasirandom graph with density at least $p$.
\end{itemize}
Then  $H$ has a perfect $K_r^{(3)}$-packing.  Moreover, we show that neither hypotheses above can be
weakened, so in this sense our result is tight.  A similar conclusion for this special case can be
proved by Keevash's hypergraph blowup lemma, with a slightly stronger hypothesis on $H$.
\end{abstract}

\section{Introduction} 
\label{sec:Introduction}

A $k$-uniform hypergraph $H$ ($k$-graph for short) is a collection of $k$-element subsets (edges) of
a vertex set $V(H)$. For a $k$-graph $H$ and a subset $S$ of vertices of size at most $k-1$, define
the $(k-|S|)$-graph $N_H(S) := \{ T \subseteq V(H) - S : T \cup S \in H\}$.  Also, let $d_H(S) =
|N_H(S)|$.  When $S = \{x\}$, we write $N_H(x)$ and $d_H(x)$.  The \emph{minimum $\ell$-degree} of
$H$, written $\delta_\ell(H)$, is the minimum of $d_H(S)$ taken over all $\ell$-sets $S \in
\binom{V(H)}{\ell}$.  The \emph{minimum codegree} of $H$ is $\delta_{k-1}(H)$ and the \emph{minimum
degree} is $\delta(H) = \delta_1(H)$.  The complete $k$-graph on $r$ vertices, denoted $K^{(k)}_r$
(or sometimes just $K_r$) is the $k$-graph with vertex set $[r]$ and all $\binom{r}{k}$ edges.  If
$H$ is a $k$-graph and $x \in V(H)$, the \emph{link of $x$}, written $L_H(x)$, is the $(k-1)$-graph
whose vertex set is $V(H) - \{x\}$ and whose edge set is $N_H(x)$.  We write $v(H)$ for $|V(H)|$.

Let $G$ and $F$ be $k$-graphs.  We say that $G$ has a \emph{perfect $F$-packing} if the vertex set
of $G$ can be partitioned into copies of $F$.  Minimum degree conditions that force perfect
$F$-packings in graphs have a long history and have been well studied~\cite{pp-alon96, pp-hajnal70,
pp-komlos01, pp-kuhn09}.  In the past decade there has been substantial interest in extending these
result to $k$-graphs~\cite{pp-czgrinow13, pp-han09, pp-keevash13, pp-kahn11, pp-kahn13,
pp-kuhn06-cherry,pp-kuhn13-fractional, pp-kuhn13, pp-lo-kfour, pp-lo13, pp-markstrom11,
pp-pikhurko08, pp-rodl09, pp-treglown09, pp-treglown13}.  Despite this activity many basic questions
in this area remain open.  For example, for $k \geq 5$ the minimum degree threshold which forces a
perfect matching in $k$-graphs is not known.

A key ingredient in the proofs of most of the previously cited results are specially designed
random-like or quasirandom properties of $k$-graphs that imply the existence of perfect
$F$-packings.  There is a rather well-defined notion of quasirandomness for graphs that originated
in early work of Thomason~\cite{qsi-thomason87,qsi-thomason87-2} and
Chung-Graham-Wilson~\cite{qsi-chung89}.  These graph quasirandom properties, when generalized to
$k$-graphs, provide a rich structure of inequivalent hypergraph quasirandom properties
(see~\cite{hqsi-lenz-poset12, hqsi-towsner14}).  In~\cite{pp-lenz14}, the authors studied in detail
the packing problem for the simplest of these quasirandom properties, the so-called weak hypergraph
quasirandomness.  A hypergraph is \emph{linear} if every two edges share at most one vertex.
Results of \cite{pp-lenz14} showed that weak hypergraph quasirandomness and an obvious minimum
degree condition suffices to obtain perfect $F$-packings for all linear $F$, but the result does not
hold for certain $F$ that are very close to being linear.

In this paper, we address the packing problem for the other quasirandom properties.  A special case
of our result identifies what hypergraph quasirandom property and what condition on the link of each
vertex is required in order to be able to guarantee a perfect $K^{(k)}_r$-packing for all $r$ (which
implies a perfect $F$-packing for all $F$). The quasirandom property naturally has great resemblance
to those used in the various (strong) hypergraph regularity lemmas.   Keevash's hypergraph blowup
lemma~\cite{pp-keevash11} has as a corollary that the super-regularity of complexes implies the
existence of perfect packings, but our main result below (Theorem~\ref{thm:cdpack}) shows that a
weaker notion of quasirandomness is enough to obtain perfect packings of complete hypergraphs.  In
fact, we are able to do more: for many of the hypergraph quasirandom properties that have been
studied previously in the literature, we give a class of hypergraphs $F$ for which we can find a
perfect packing.  Before stating Theorem~\ref{thm:cdpack}, we need to define these notions of
hypergraph quasirandomness.

\subsection{Notions of Hypergraph Quasirandomness} 
\label{sub:definitions}

Our definitions are closely related to the definitions by Towsner~\cite{hqsi-towsner14}, which gives
the most general treatment of hypergraph quasirandomness.

\begin{definition}
  Let $X$ be a finite set and let $2^X = \{A : A \subseteq X\}$.  An \emph{antichain} is an
  $\mathcal{I} \subseteq 2^X$ such that $A \subsetneq B$ for all $A, B \in \mathcal{I}$.  A
  \emph{full antichain} is an antichain $\mathcal{I} \subseteq 2^X$ such that $|\mathcal{I}| \geq 2$
  and for all $x \in X$, there exists $I \in \mathcal{I}$ with $x\in I$.
\end{definition}

\begin{definition}
  Let $k \geq 1$, let $\mathcal{I} \subseteq 2^{[k]}$ be an antichain, and let $H$ be a $k$-graph.
  An \emph{$\mathcal{I}$-layout in $H$} is a tuple of uniform hypergraphs $\Lambda = (\lambda_I)_{I
  \in \mathcal{I}}$ where $\lambda_I$ is an $|I|$-uniform hypergraph on vertex set $V(H)$.  If
  $\Lambda$ is an $\mathcal{I}$-layout, then the \emph{$k$-cliques of $\Lambda$}, denoted
  $K_k(\Lambda)$, is the set of all vertex tuples $(x_1, \dots, x_k)$ such that $x_1, \dots, x_k$
  are distinct vertices and for each $I \in \mathcal{I}$, $\{x_i : i \in I\} \in \lambda_I$.  In an
  abuse of notation, we will denote by $H \cap K_k(\Lambda)$ the $k$-tuples $(x_1,\dots,x_k)$ such
  that $(x_1,\dots,x_k) \in K_k(\Lambda)$ and $\{x_1,\dots,x_k\} \in H$.
\end{definition}

We now are ready to define hypergraph quasirandomness.

\begin{definition}
  Let $0 < \mu,p < 1$.  A $k$-graph $H$ satisfies \discIk if for every $\mathcal{I}$-layout
  $\Lambda$,
  \begin{align*}
    |H \cap K_k(\Lambda)| \geq p |K_k(\Lambda)|  - \mu n^k.
  \end{align*}
  The stronger property \discIstrong
  stipulates that for every $\mathcal{I}$-layout $\Lambda$,
  \begin{align*}
    \Big| |H \cap K_k(\Lambda)| - p|K_k(\Lambda)| \Big| \leq \mu n^k.
  \end{align*}
\end{definition}

\begin{example}
  Let $k = 3$ and $\mathcal{I} = \{\{1,2\}, \{2, 3\}\}$.  A $3$-graph $H$ satisfies
  \discthree if for every two graphs $\lambda_{12}$ and $\lambda_{23}$ with vertex set $V(H)$, the
  number of tuples $(x,y,z)$ with $\{x,y,z\} \in H$, $xy \in \lambda_{12}$, and $yz \in
  \lambda_{23}$ is at least $p|K_3(\lambda_{12},\lambda_{23})| - \mu n^3$, where
  $K_3(\lambda_{12},\lambda_{23})$ is the set of tuples $(x,y,z)$ with $xy \in \lambda_{12}$ and $yz
  \in \lambda_{23}$.
\end{example}

Several special cases of this definition deserve mention, since essentially all previously studied
hypergraph quasirandomness properties are related to \discIk for some
$\mathcal{I}$.

\begin{itemize}
  \item When $\mathcal{I} = \{ \{1\}, \dots, \{k\}\}$, then \discIk is exactly the property
    $(p,\frac{\mu}{k!})$-dense from~\cite{pp-lenz14} and is closely related to weak quasirandomness
    studied in~\cite{hqsi-conlon12, hqsi-dellamonica11, hqsi-kohayakawa10, hqsi-shapira12}. 

  \item More generally, when $\mathcal{I}$ is a partition the property \discIstrong is essentially
    the property \texttt{Expand[$\pi$]} studied in~\cite{hqsi-lenz-quasi12,
    hqsi-lenz-quasi12-nonregular, hqsi-lenz-poset12}.  In particular, when $\mathcal{I} =
    \{\{1,\dots,k-1\},\{k\}\}$, then \discIk is essentially equivalent to the property considered
    recently by Keevash (the property called ``typical'' in~\cite{design-keevash14}) in his recent
    proof of the existence of designs.

  \item When $\mathcal{I} = \binom{[k]}{\ell}$, then \discIstrong is closely related to the property
    \texttt{CliqueDisc}[$\ell$] studied in~\cite{hqsi-chung12, hqsi-chung90, hqsi-chung90-2,
    hqsi-chung91, hqsi-chung92, hqsi-kohayakawa02, hqsi-lenz-poset12}.

  \item When $\mathcal{I} = \{ I \in \binom{[k]}{k-1} : \{1,\dots,\ell\} \subseteq I\}$, then
    \discIstrong is essentially the same as the property \texttt{Deviation}[$\ell$] studied
    in~\cite{hqsi-chung90-2, hqsi-chung91, hqsi-chung90, hqsi-kohayakawa02, hqsi-lenz-poset12}.

  \item Finally, note that \discp{\{\emptyset\}} is equivalent to $|H| \geq p
    \binom{v(H)}{k} - \frac{\mu}{k!} n^k$, since $K_k(\{\emptyset\})$ is the set of all ordered
    $k$-tuples of distinct vertices.
\end{itemize}
  
\begin{definition}
  Let $\mathcal{I} \subseteq 2^{[k]}$ be an antichain.  A $k$-graph $F$ is
  \emph{$\mathcal{I}$-adapted} if there exists an ordering $E_1, \dots, E_m$ of the edges of $F$ and
  bijections $\phi_i : E_i \rightarrow [k]$ such that for each $1 \leq j < i \leq m$, the following
  holds: there exists an $I \in \mathcal{I}$ with $\{ \phi_i(x) : x \in E_j \cap E_i\} \subseteq I
  \in \mathcal{I}$.  In words, $F$ is \emph{$\mathcal{I}$}-adapted if the set of labels assigned to
  $E_i$ which appear on $E_j \cap E_i$ is a subset of a set in $\mathcal{I}$.

  Let $\mathcal{I} \subseteq 2^{[k]}$ and $\mathcal{J} \subseteq 2^{[k-1]}$ be antichains.  A
  $k$-graph $F$ is \emph{$(\mathcal{I},\mathcal{J})$-adapted} if $F$ is $\mathcal{I}$-adapted and
  there exists $x \in V(F)$, an ordering $E_1, \dots, E_m$ of the edges of $F$, and bijections
  $\psi_i : E_i \rightarrow [k]$ such that for all $1 \leq j < i \leq m$, the following holds.
  \begin{itemize}
    \item If $x \notin E_i$ then there exists $I \in \mathcal{I}$ with
      $\{\psi_i(y) : y \in E_j \cap E_i\} \subseteq I$. 

    \item  If $x \in E_i$ then $\psi_i(x) = k$ and there exists $J \in \mathcal{J}$ with $\{\psi_i(y) : y \in
      E_j \cap E_i, y \neq x\} \subseteq J$.
  \end{itemize}
\end{definition}

\subsection{Our Results} 
\label{sub:our-results}

The following is our main result.

\begin{thm} \label{thm:cdpack}
  Let $k \geq 2$, $\mathcal{I} \subseteq 2^{[k]}$ be a full antichain, $\mathcal{J} \subseteq
  2^{[k-1]}$, and $0 < \alpha,p < 1$.  For every $(\mathcal{I},\mathcal{J})$-adapted $k$-graph $F$,
  there exists $\mu > 0$ and $n_0$ so that the following holds.  Let $H$ be an $n$-vertex $k$-graph
  where $n \geq n_0$ and $v(F) | n$.  Suppose that $H$ satisfies \discIk and that $L_H(x)$ satisfies
  \discImone for all $x \in V(H)$.  Then $H$ has a perfect $F$-packing.
\end{thm}

It is straightforward to see that if $\mathcal{I}$ and $\mathcal{I}'$ are such that for every $I'
\in \mathcal{I}'$, there exists $I \in \mathcal{I}$ with $I' \subseteq I$, then \discp{\mathcal{I}}
$\Rightarrow$ \discp{\mathcal{I'}}. Also, if $\mathcal{I} = \binom{[k]}{k-1}$  and $\mathcal{J} =
\binom{[k-1]}{k-2}$, then every $F$ is $(\mathcal{I},\mathcal{J})$-adapted.  Thus to find the
weakest quasirandom condition to apply Theorem~\ref{thm:cdpack} to a given $k$-graph $F$, one should
find the minimal $\mathcal{I}$ and $\mathcal{J}$ for which $F$ is
$(\mathcal{I},\mathcal{J})$-adapted.  For example, if $C = \{abc, bcd, def, aef\}$, then $C$ is
$(\mathcal{I},\mathcal{J})$-adapted where $\mathcal{I} = \{\{1,2\},\{3\}\}$ and $\mathcal{J} =
\{\emptyset\}$ (let $x = a$ and order the edges which contain $a$ first).

As mentioned above, special cases of \discIk correspond to previously studied quasirandom properties
so that Theorem~\ref{thm:cdpack} generalizes several previous results.

\begin{itemize}
  \item Let $k = 2$.  The only full antichain is $\mathcal{I} = \{\{1\}, \{2\}\}$.  For this
    $\mathcal{I}$, all graphs $F$ are $(\mathcal{I},\mathcal{J})$-adapted if $\mathcal{J} =
    \{\emptyset\}$.  To see this, pick $x \in V(F)$ and place all edges incident to $x$ first in the
    ordering for the definition of $(\mathcal{I},\mathcal{J})$-adapted. Now the property \disctwo
    just states that $G$ is quasirandom (in fact only ``one-sided'' quasirandom).  Also, the
    condition ``$L_H(x)$ satisfies \discone for every $x \in V(H)$'' is equivalent to the condition
    that $\delta(H) \geq (\alpha - \mu)(n-1)$.  To see this, recall from before that if $H'$ is an
    $r$-graph the property ``$H'$ satisfies \discone is equivalent to the property that $|H'|\geq
    \alpha \binom{v(H')}{r} - \frac{\mu}{r!} v(H')^{r}$.  Thus Theorem~\ref{thm:cdpack} for $k =
    2$ states that if $G$ is an $n$-vertex quasirandom graph, $v(F) | n$, and $\delta(G) \geq
    (\alpha - \mu)(n-1)$, then $G$ has a perfect $F$-packing.  This fact is a simple consequence of
    the blowup lemma of Koml\'os-S\'ark\"ozy-Szemer\'edi~\cite{pp-komlos97}.

  \item For $k \geq 2$ with $\mathcal{I}$ a partition into singletons, we obtain exactly
    \cite[Theorem 3]{pp-lenz14}.  In this case, \discIk is equivalent to $(p,\frac{\mu}{k!})$-dense
    from~\cite{pp-lenz14}, an $\mathcal{I}$-adapted $k$-graph is a linear $k$-graph, and one can
    take $\mathcal{J} = \{\emptyset\}$.  Similar to the previous paragraph, the condition ``$L_H(x)$
    satisfies \disca{\{\emptyset\}} for every $x \in V(H)$'' is equivalent to the condition that
    $\delta(H) \geq \alpha \binom{v(H)-1}{k-1} - \frac{\mu}{(k-1)!} v(H)^{k-1}$. 

  \item If $\mathcal{I} = \binom{[k]}{k-1}$ and $\mathcal{J} = \binom{[k-1]}{k-2}$ then every
    $k$-graph $F$ is $(\mathcal{I},\mathcal{J})$-adapted.  Thus Theorem~\ref{thm:cdpack} implies the
    following corollary.
    \begin{cor} \label{cor:main}
      Fix $2 \leq k \leq r$.  For every $0 < \alpha, p < 1$, there exists $\mu > 0$ and $n_0$ such
      that the following holds.  Let $H$ be an $n$-vertex $k$-graph with $n \geq n_0$ and $r|n$.  If
      $H$ satisfies \cdk and $L_H(x)$ satisfies \cdkmone for every $x \in V(H)$, then $H$ has a
      perfect $K^{(k)}_r$-packing.
    \end{cor}
    Keevash's hypergraph blowup lemma~\cite{pp-keevash11} also guarantees perfect
    $K^{(k)}_r$-packings under certain regularity conditions, however the hypotheses of
    Corollary~\ref{cor:main} are slightly weaker.  Indeed, the main extra requirement that
    \cite{pp-keevash11} places on $H$ is \cite[Definition 3.16 part \textit{(iii)}]{pp-keevash11};
    translated into our language, for $3$-graphs this property says roughly that for every $x \in
    V(H)$, if $W$ is a set of triples where each triple contains some pair from $L_H(x)$, then $|H
    \cap W| \approx p |W|$.
\end{itemize}

Next, we investigate if either of the conditions \discIk or \discImone in the links from
Theorem~\ref{thm:cdpack} can be weakened.  This question was studied by the authors~\cite{pp-lenz14}
in detail when $\mathcal{I}$ is a partition, and it turns out that for certain non-linear $F$ it is
possible to weaken the conditions (see~\cite{pp-lenz14} for details).  Most likely, the
constructions and results from~\cite{pp-lenz14} can be generalized to all $\mathcal{I}$.  In this
paper, we focus only on the case $\mathcal{I} = \binom{[k]}{k-1}$ and $\mathcal{J} =
\binom{[k-1]}{k-2}$, which corresponds to the condition required for perfect $K^{(k)}_r$-packings.
In this case, neither condition can be weakened, so that Theorem~\ref{thm:cdpack} cannot be improved
in general.

\begin{prop} \label{prop:constr-fail-kminusone}
  For every $k \geq 3$ there exists an $r$ (depending only on $k$) such that the following holds.
  Let $\alpha = p = \frac{k-1}{k}$ and let $\mathcal{I} \subseteq 2^{[k]}$ be a full antichain where
  $\mathcal{I} \neq \binom{[k]}{k-1}$.  For every $\mu > 0$, there exists $n_0$ such that for all $n
  \geq n_0$ there exists an $n$-vertex $k$-graph $H$ which
  \begin{itemize}
    \item satisfies \discIk,
    \item fails \cdk,
    \item for every $x \in V(H)$ the link $L_H(x)$ satisfies \cdkmone,
    \item has no copy of $K_r$ (so no perfect $K_r$-packing).
  \end{itemize}
\end{prop}

\begin{prop} \label{prop:constr-fail-links}
  For every $k \geq 3$ there exists an $r$ (depending only on $k$) such that the following holds.
  Let $\alpha = p = \frac{k-1}{k}$ and let $\mathcal{J} \subseteq 2^{[k-1]}$ be a full antichain
  where $\mathcal{J} \neq \binom{[k-1]}{k-2}$.  For every $0 < \mu, p < 1$, there exists $n_0$ such
  that for all $n \geq n_0$ with $r|n$, there exists an $n$-vertex $k$-graph $H$ which
  \begin{itemize}
    \item satisfies \cdk,
    \item for every $x \in V(H)$ the link $L_H(x)$ satisfies \disca{\mathcal{J}},
    \item there exists $x \in V(H)$ such that the link $L_H(x)$ fails \cdkmone,
    \item has no perfect $K_r$-packing.
  \end{itemize}
\end{prop}

The remainder of this paper is organized as follows.  In Sections~\ref{sec:tools}
and~\ref{sec:embedding} we discuss the two main tools needed for the proof of Theorem~\ref{thm:cdpack},
in Section~\ref{sec:packing} we prove Theorem~\ref{thm:cdpack}, and finally in
Section~\ref{sec:Constructions} we explain the constructions which prove
Propositions~\ref{prop:constr-fail-kminusone} and~\ref{prop:constr-fail-links}.

\section{Absorbing Sets} 
\label{sec:tools}

One of the main tools for our proof of Theorem~\ref{thm:cdpack} is the absorbing technique of
R\"odl-Ruci\'nski-Szemer\'edi~\cite{pp-rodl09}.  We will use the following absorbing lemma
from~\cite{pp-lenz14} without modification.

\begin{definition}
  Let $F$ and $H$ be $k$-graphs and let $A, B \subseteq V(H)$.  We say that $A$
  \emph{$F$-absorbs} $B$ or that $A$ is an \emph{$F$-absorbing set} for $B$ if both $H[A]$ and $H[A
  \cup B]$ have perfect $F$-packings.  When $F$ is a single edge, we say that $A$
  \emph{edge-absorbs} $B$.
\end{definition}

\begin{definition}
  Let $F$ and $H$ be $k$-graphs, $\epsilon > 0$, and $a$ and $b$ be multiples of $v(F)$.  We say
  that $H$ is \emph{$(a,b,\epsilon,F)$-rich} if for all $B \in \binom{V(H)}{b}$ there are at least
  $\epsilon n^a$ sets in $\binom{V(H)}{a}$ which $F$-absorb $B$.
\end{definition}

\begin{lemma} (Absorbing Lemma, specialized version of \cite[Lemma 10]{pp-lenz14}) \label{lem:absorbing}
  Let $F$ be a $k$-graph, $\epsilon > 0$, and $a$ and $b$ be multiples of $v(F)$.  There exists an
  $n_0$ and $\omega > 0$ such that for all $n$-vertex $k$-graphs $H$ with $n \geq n_0$, the
  following holds.  If $H$ is $(a,b,\epsilon,F)$-rich, then there exists an $A \subseteq V(H)$ such
  that $a | |A|$ and $A$ $F$-absorbs all sets $C$ satisfying the following conditions: $C \subseteq
  V(H) - A$, $|C| \leq \omega n$, and $b | |C|$.
\end{lemma}

\section{Embedding Lemma} 
\label{sec:embedding}

\begin{definition}
  Let $k \geq 2$ and $0 \leq m \leq f$.  Let $F$ and $H$ be $k$-graphs with $V(F) =
  \{w_1,\dots,w_f\}$.  A \emph{labeled copy of $F$ in $H$} is an edge-preserving injection from
  $V(F)$ to $V(H)$.  A \emph{degenerate labeled copy of $F$ in $H$} is an edge-preserving map from
  $V(F)$ to $V(H)$ that is not an injection.  Let $1 \leq m \leq f$ and let $Z_1, \dots, Z_m
  \subseteq V(H)$.  Set $\inj[F \rightarrow H; w_1 \rightarrow Z_1, \dots, w_m \rightarrow Z_m]$ to
  be the number of edge-preserving injections $\psi : V(F) \rightarrow V(H)$ such that $\psi(w_i)
  \in Z_i$ for all $1 \leq i \leq m$.  If $Z_i = \{z_i\}$, we abbreviate $w_i \rightarrow \{z_i\}$
  as $w_i \rightarrow z_i$.
\end{definition}

The embedding lemma (Lemma~\ref{lem:embeddingatvertex}) proved in this section shows that if $H$
satisfies \discIk and \discImone in the links, then $H$ contains many copies of $F$ if $F$ is
$(\mathcal{I},\mathcal{J})$-adapted.  In fact, it says more: if $m$ of the vertices of $F$ are
pre-specified and $F$ satisfies the following more technical condition, then there are many copies
of $F$ using the $m$ pre-specified vertices.

\begin{definition}
  Let $k \geq 2$, $\mathcal{I} \subseteq 2^{[k]}$ and $\mathcal{J} \subseteq 2^{[k-1]}$ be
  antichains, $F$ a $k$-graph, and $s_1, \dots, s_m \in V(F)$.  We say that $F$ is
  \emph{$(\mathcal{I},\mathcal{J})$-adapted at $s_1,\dots,s_m$} if there exists an ordering $E_1,
  \dots, E_t$ of the edges of $F$ such that
  \begin{itemize}
    \item for every $i$, $|E_i \cap \{s_1,\dots,s_m\}| \leq 1$,

    \item for every $E_i$ with $E_i \cap \{s_1,\dots,s_m\} = \emptyset$, there exists a bijection
      $\phi_i : E_i \rightarrow [k]$ such that for all $j < i$, there exists $I \in \mathcal{I}$ with
      $\{\phi_i(x) : x \in E_j \cap E_i\} \subseteq I$,

    \item for every $E_i$ with $s_\ell \in E_i$, there exists a bijection $\psi_i : E_i
      \setminus\{s_\ell\} \rightarrow [k-1]$ such that for all $j < i$, there exists $J \in
      \mathcal{J}$ with $\{ \psi_i(x) : x \in E_j \cap E_i, x \neq s_\ell \} \subseteq J$.
  \end{itemize}
  Note that $m = 0$ is possible, in which case the definition is equivalent to
  $\mathcal{I}$-adapted.
\end{definition}

\begin{lemma}\label{lem:embeddingatvertex}
  Let $k \geq 2$, $0 < \alpha,\gamma,p < 1$, and $\mathcal{I} \subseteq 2^{[k]}$ and $\mathcal{J}
  \subseteq 2^{[k-1]}$ be antichains.  Let $F$ be an $f$-vertex $k$-graph with $V(F) = \{s_1, \dots,
  s_m, t_{m+1},\dots,t_f\}$.  Suppose that $F$ is $(\mathcal{I},\mathcal{J})$-adapted at
  $s_1,\dots,s_m$.  Then there exists an $n_0$ and $\mu > 0$ such that the following is true. 
  
  Let $H$ be an $n$-vertex $k$-graph with $n \geq n_0$, where $H$ satisfies \discIk.  If $m > 0$,
  then also assume that $L_H(x)$ satisfies \discImone for every vertex $x \in V(H)$.  Let $y_1,
  \dots, y_m \in V(H)$ be distinct and let $V_{m+1},\dots,V_{f} \subseteq V(H)$.  Then
  \begin{align*}
    \inj[F \rightarrow H;s_1 \rightarrow y_1, \dots, &s_m \rightarrow y_m, t_{m+1} \rightarrow
        V_{m+1},\dots,t_f \rightarrow V_f] \\
    &\geq 
    \alpha^{d_F(s_1)} \cdots \alpha^{d_F(s_m)} p^{|F|-\sum d_F(s_i)} |V_{m+1}| \cdots |V_f| - \gamma n^{f-m}.
  \end{align*}
\end{lemma}

\begin{proof} 
We first prove the lemma under the additional assumption that the sets $V_{m+1},\dots,V_f$ are
pairwise disjoint.  This is proved by induction on $|F|$.  If $|F| = 0$, then
\begin{align*}
  \inj[F \rightarrow H;s_1 \rightarrow y_1, \dots, s_m \rightarrow y_m, t_{m+1} \rightarrow
  V_{m+1},\dots,t_f \rightarrow V_f] &\geq \prod_{i=m+1}^{f} \left( |V_i| - f \right) \\
                                     &\geq \alpha^0 p^0 \prod_{i=m+1}^f |V_i| - \gamma n^{f-m}
\end{align*}
for large $n$.  So assume $F$ has at least one edge and let $E$ be the last edge in an ordering of the
edges of $F$ which witness that $F$ is $(\mathcal{I},\mathcal{J})$-adapted at $s_1,\dots,s_m$.
(Recall that if $m = 0$ then $(\mathcal{I},\mathcal{J})$-adapted at $s_1,\dots,s_m$ is equivalent to
$\mathcal{I}$-adapted.)

Let $F_*$ be the hypergraph formed by deleting all vertices of $E$ from $F$.  Let $F_-$ be the
hypergraph formed by removing the edge $E$ from $F$ but keeping the same vertex set.  Let $Q_*$ be
an injective edge-preserving map $Q_* : V(F_*) \rightarrow V(H)$ where $Q_*(s_i) = y_i$ for $1 \leq
i \leq m$ and $Q_*(t_j) \in V_j$ for $t_j \notin E$.  There are two cases.

\medskip \textbf{Case 1:} $E \cap \{s_1,\dots,s_m\} = \emptyset$.  Let $\phi : E \rightarrow [k]$ be
the bijection from the definition of $(\mathcal{I},\mathcal{J})$-adapted at $s_1,\dots,s_m$ and
assume the vertices of $F$ are labeled such that $E = \{t_{m+1},\dots,t_{m+k}\}$, where
$\phi(t_{m+i}) = i$. For each $I \in \mathcal{I}$, define an $|I|$-uniform hypergraph
$\lambda_{I,Q_*}$ with vertex set $V(H)$ as follows.  Let $I = \{i_1,\dots,i_{|I|}\}$.  Make
$\{z_{i_1},\dots,z_{i_{|I|}}\} \in \binom{V(H)}{|I|}$ a hyperedge of $\lambda_{I,Q_*}$ if $z_{i_j}
\in V_{m + i_j}$ for all $j$ and when the map $Q_*$ is extended to map $t_{i_j}$ to $z_{i_j}$ for
all $j$, this extended map is an edge-preserving map from $F_-[V(F_*) \cup
\{t_{i_1},\dots,t_{i_{|I|}}\}]$ to $H$.  More informally, $\lambda_{I,Q_*}$ consists of all
$|I|$-sets to which $Q_*$ can be extended to produce a copy of $F_*$ together with the vertices of
$E$ indexed by $I$.  Let $\Lambda_{Q_*} = (\lambda_{I,Q_*})_{I \in \mathcal{I}}$.

Now, if $(z_{m+1},\dots,z_{m+k})$ is a $k$-tuple in $K_k(\Lambda_{Q_*})$, then the map $Q_*$ can be
extended to map $t_j$ to $z_j$ for $m+1 \leq j \leq m+k$ to produce an edge-preserving map from
$F_-$ to $H$.  To see this, let $E'$ be an edge of $F_-$.  Since $E$ is the last edge in the
ordering, if $E' \cap E = \{t_{j_1},\dots,t_{j_r}\}$ then there exists some $I \in \mathcal{I}$ with
$\{j_1,\dots,j_r\} \subseteq I$ since $F$ is $\mathcal{I}$-adapted.  Since $(z_{m+1},\dots,z_{m+k})$
is a $k$-clique, $\{z_{m+i} : i \in I\} \in \lambda_{I,Q_*}$.  This implies that there is some
permutation $\eta$ of $I$ such that extending $Q_*$ to map $t_{m+i}$ to $z_{m+\eta(i)}$ produces an
edge-preserving map.  Since the $V_{m+i}$s are pairwise disjoint and $z_{m+i} \in V_{m+i}$ for all
$i \in I$, $\eta$ must be the identity permutation, i.e.\ extending the map $Q_*$ to map $t_{m+i}$
to $z_{m+i}$ for all $i \in I$ produces an edge-preserving map.  Thus extending the map $Q_*$ to map
$t_{j_p}$ to $z_{j_p}$ for all $p$ is an edge-preserving map and $E'$ is one of the preserved edges.
Finally, since the $V_j$s are disjoint, each $k$-tuple in $K_k(\Lambda_{Q_*})$ corresponds to
exactly one labeled copy of $F_-$ in $H$ which extend $Q_*$ with $t_j$ mapped into $V_j$ for all
$j$.  Similarly, $|H \cap K_k(\Lambda_{Q_*})|$ is exactly the number of labeled copies of $F$ in $H$
which extend $Q_*$ with $t_j$ mapped into $V_j$ for all $j$.  Thus,
\begin{align}
  \inj[F \rightarrow H; s_1\rightarrow y_1,\dots,s_m\rightarrow y_m, t_{m+1}\rightarrow V_{m+1},
  \dots, t_f\rightarrow V_f] &= \sum_{Q_*} |H \cap K_k(\Lambda_{Q_*})| \nonumber \\
  \inj[F_- \rightarrow H; s_1\rightarrow y_1,\dots,s_m\rightarrow y_m, t_{m+1}\rightarrow V_{m+1},
  \dots, t_f\rightarrow V_f] &= \sum_{Q_*} |K_k(\Lambda_{Q_*})|. \label{eq:countFminus}
\end{align}
Since $H$ satisfies \discIk,
\begin{align}
  \inj[F \rightarrow H;s_1\rightarrow y_1,\dots,s_m\rightarrow y_m, &t_{m+1}\rightarrow V_{m+1},
  \dots, t_f\rightarrow V_f] \nonumber \\
  &\geq \sum_{Q_*} \left( p |K_k(\Lambda_{Q_*})| - \mu
  n^{k}\right) \nonumber \\
  &\geq p\sum_{Q_*} |K_k(\Lambda_{Q_*})| - \mu n^{f-m}, \label{eq:countFbyQsum}
\end{align}
where the last inequality is because there are at most $n^{f-m-k}$ maps $Q_*$, since $F_*$ has $f-k$
vertices and $s_i \in V(F_*)$ must map to $y_i$.  Combining \eqref{eq:countFminus} and
\eqref{eq:countFbyQsum} and then applying induction,
\begin{align*}
  \inj[F \rightarrow H; &s_1\rightarrow y_1,\dots,s_m\rightarrow y_m, t_{m+1}\rightarrow V_{m+1},
  \dots, t_f\rightarrow V_f] \\
  &\geq p \inj[F_- \rightarrow H; s_1\rightarrow y_1,\dots,s_m\rightarrow y_m, t_{m+1}\rightarrow
  V_{m+1}, \dots, t_f\rightarrow V_f] - \mu
  n^{f-m} \\
  &\geq p \left( \alpha^{\sum d(s_i)} p^{|F|-1-\sum d(s_i)}|V_{m+1}|\cdots|V_f|  - \gamma n^{f-m}  \right) - \mu
  n^{f-m}.
\end{align*}
Let $\mu = (1-p)\gamma$ so that the proof of this case complete.

\medskip

\textbf{Case 2:} $s_\ell \in E$.  (Since $F$ is $(\mathcal{I},\mathcal {J})$-adapted at
$s_1,\dots,s_m$, at most one vertex $s_\ell$ can be in $E$.)  Let $\psi : E \setminus \{s_\ell\}
\rightarrow [k-1]$ be the bijection from the definition of $(\mathcal{I},\mathcal{J})$-adapted at
$s_1,\dots,s_m$ and assume the vertices of $E$ are labeled such that $E = \{s_\ell,
t_{m+1},\dots,t_{m+k-1}\}$ where $\psi(t_{m+j}) = j$.  This case is very similar to the previous
case, except we will use \discImone in the link of $y_\ell$.  For each $J \in \mathcal{J}$, define a
$|J|$-uniform hypergraph $\lambda_{J,Q_*}$ with vertex set $V(H)$ as follows.  Let $J =
\{j_1,\dots,j_{|J|}\}$.  Make $\{z_{j_1},\dots,z_{j_{|J|}}\}$ a hyperedge of $\lambda_{J,Q_*}$ if
$z_{j_r} \in V_{j_r}$ for all $r$ and extending the map $Q_*$ to map $s_\ell$ to $y_\ell$ and
mapping $t_{j_r}$ to $z_{j_r}$ for all $r$ produces an edge-preserving map.  Let $\Lambda_{Q_*} =
(\lambda_{J,Q_*})_{J \in \mathcal{J}}$. Similar to before, if $(z_{m+1},\dots,z_{m+k-1})$ is a
$(k-1)$-tuple in $K_{k-1}(\Lambda_{Q_*})$, then the map $Q_*$ can be extended to map $s_\ell$ to
$y_\ell$ and map $t_{i}$ to $z_i$ for $m+1 \leq i \leq m+k-1$ to produce an edge-preserving map from
$F_-$ to $H$.  Thus $|K_{k-1}(\Lambda_{Q_*})|$ is exactly the number of labeled copies of $F_-$ in
$H$ which extend $Q_*$.  Similarly, $|L_H(y_\ell) \cap K_{k-1}(\Lambda_{Q_*})|$ is exactly the
number of labeled copies of $F$ in $H$ which extend $Q_*$.

Now formulas similar to \eqref{eq:countFminus} and \eqref{eq:countFbyQsum} and the fact that
$L_H(y_\ell)$ satisfies \discImone completes this case.  This concludes the proof of the lemma if
the sets $V_{m+1},\dots,V_f$ are pairwise disjoint.

Now assume that the sets $V_{m+1},\dots,V_f$ are not necessarily pairwise disjoint.  Let
$\mathcal{P} = \{ (P_{m+1},\dots,P_f) : P_{m+1},\dots,P_f \text{ is a partition of } V(H) \}$ so
that $|\mathcal{P}| = (f-m)^{n}$.  Now
\begin{align*}
  \inj[F \rightarrow H; s_1\rightarrow y_1,\dots,s_m\rightarrow y_m, t_{m+1}\rightarrow &V_{m+1},
  \dots, t_f\rightarrow V_f] \\
  = \frac{1}{(f-m)^{n-f+m}} \sum_{(P_{m+1},\dots,P_f) \in \mathcal{P}} 
  \inj[F \rightarrow H; &s_1\rightarrow y_1,\dots,s_m\rightarrow y_m, \\ &t_{m+1}\rightarrow V_{m+1}
  \cap P_{m+1}, \dots, t_f\rightarrow V_f \cap P_f].
\end{align*}
Indeed, each labeled copy of $F$ of the right form will be counted exactly $(f-m)^{n-f+m}$ times by
the sum over all partitions, since the images of $t_{m+1},\dots,t_f$ must map into the cooresponding
part of the partition and all other vertices of $H$ can be distributed to any of the parts of the
partition.  Let $\delta = \alpha^{d_F(s_1)} \cdots \alpha^{d_F(s_m)} p^{|F| - \sum d_F(s_i)}$. Since
$V_{m+1} \cap P_{m+1}, \dots, V_f \cap P_f$ are pairwise disjoint,
\begin{align*}
  \inj[F \rightarrow H; &s_1\rightarrow y_1,\dots,s_m\rightarrow y_m, t_{m+1}\rightarrow V_{m+1},
  \dots, t_f\rightarrow V_f] \\
  &\geq \frac{1}{(f-m)^{n-f+m}} \sum_{(P_{m+1},\dots,P_f) \in \mathcal{P}} 
  \left( \delta |V_{m+1} \cap P_{m+1}| \cdots |V_f \cap P_f| - \gamma n^{f-m} \right) \\
  &= \delta |V_{m+1}| \cdots |V_f| - \frac{\gamma n^{f-m} |\mathcal{P}|}{(f-m)^{n-f+m}}
  \geq \delta |V_{m+1}| \cdots |V_f| - \gamma n^{f-m}.
\end{align*}
\end{proof} 

\section{Packing $(\mathcal{I},\mathcal{J})$-adapted hypergraphs} 
\label{sec:packing}

In this section we prove Theorem~\ref{thm:cdpack}.  The proof has several stages: we first prove
that the quasirandom conditions on $H$ imply that $H$ is rich, then we use Lemma~\ref{lem:absorbing}
to set aside a vertex set $A$ which can absorb all reasonably sized sets, next we use the embedding
lemma (Lemma~\ref{lem:embeddingatvertex}) to produce an almost perfect packing in $H - A$, and
finally we use the properties of $A$ to absorb the remaining vertices.

\subsection{Richness} 
\label{sub:Richness}

In this subsection, we prove that the conditions on $H$ in Theorem~\ref{thm:cdpack} imply that $H$ is
$(f^2-f, f, \epsilon, F)$-rich, where $f = v(F)$.

\begin{lemma} \label{lem:richness}
  Let $k \geq 2$, $\mathcal{I} \subseteq 2^{[k]}$ be a full antichain, and $\mathcal{J} \subseteq
  2^{[k-1]}$ an antichain.  Let $F$ be an $(\mathcal{I},\mathcal{J})$-adapted $k$-graph with $f$
  vertices. For every $0 < \alpha,p < 1$, there exists $\mu,\epsilon > 0$ and $n_0$ so that the
  following holds.  Let $H$ be an $n$-vertex $k$-graph where $n \geq n_0$. Also, assume that $H$
  satisfies \discIk and that $L_H(z)$ satisfies \discImone for every vertex $z \in V(H)$.  Then $H$
  is $(f^2 - f, f, \epsilon, F)$-rich.
\end{lemma}

\begin{proof} 
Let $a = f(f-1)$ and $b = f$.  Our task is to come up with an $\epsilon > 0$ such that for large $n$
and all $B \in \binom{V(H)}{b}$, there are at least $\epsilon n^a$ vertex sets of size $a$ which
$F$-absorb $B$; we will define $\epsilon$ and $\mu$ later.   Let $V(F) = \{w_0,\dots,w_{f-1}\}$,
where $w_0$ is the special vertex in the definition that $F$ is $(\mathcal{I},\mathcal{J})$-adapted.

Next, form the following $k$-graph $F'$.  Let
\begin{align*}
  V(F') = \{ x_{i,j} : 0 \leq i, j \leq f-1 \}.
\end{align*}
(We think of the vertices of $F'$ as arranged in a grid with $i$ as the row and $j$ as the column.)
Form the edges of $F'$ as follows: for each fixed $1 \leq i \leq f-1$, let
$\{x_{i,0},\dots,x_{i,f-1}\}$ induce a copy of $F$ where $x_{i,j}$ is mapped to $w_j$.  Similarly,
for each fixed $0 \leq j \leq f-1$, let $\{x_{0,j},\dots,x_{f-1,j}\}$ induce a copy of $F$ where
$x_{i,j}$ is mapped to $w_{i}$.  Note that we therefore have a copy of $F$ in each column and a copy
of $F$ in each row besides the zeroth row.

Now fix $B = \{b_0, \dots, b_{f-1}\} \subseteq V(H)$; we want to show that $B$ is $F$-absorbed by many
$a$-sets.  Note that any labeled copy of $F'$ in $H$ which maps $x_{0,0} \rightarrow b_0, \dots,
x_{0,f-1} \rightarrow b_{f-1}$ produces an $F$-absorbing set for $B$ as follows.  Let $Q : V(F')
\rightarrow V(H)$ be an edge-preserving injection where $Q(b_j) = x_{0,j}$ (so $Q$ is a labeled copy
of $F'$ in $H$ where the set $B$ is the zeroth row of $F'$).  Let $A = \{Q(x_{i,j}) : 1 \leq i \leq
f-1, 0 \leq j \leq f-1\}$ consist of all vertices in rows $1$ through $f-1$.  Then $A$ has a perfect
$F$-packing consisting of the copies of $F$ on the rows, and $A \cup B$ has a perfect $F$-packing
consisting of the copies of $F$ on the columns.  Therefore, $A$ $F$-absorbs $B$.

To complete the proof, we therefore just need to use Lemma~\ref{lem:embeddingatvertex} where $m = f$
and $s_1 = x_{0,0},\dots,s_f = x_{0,f-1}$ to show that there are many copies of $F'$ with $B$ as the
zeroth row.  To do so, we need to show that $F'$ is $(\mathcal{I},\mathcal{J})$-adapted at
$s_1,\dots,s_m$.  Indeed, consider the following ordering of edges of $F'$.  First, list the edges
of $F'$ in the first column, then the edges of $F'$ in the second column, and so on until the $k$th
column.  Next, list the edges of $F'$ in the first row, then the second row, and so on until the
$(k-1)$st row.  Within each row or column, list the edges in the ordering given in the definition of
$F$ being $(\mathcal{I},\mathcal{J})$-adapted.  For the bijections $\phi$ or $\psi$, use the same
bijection as in the definition of $F$ being $(\mathcal{I},\mathcal{J})$-adapted.  Now consider $E_i,
E_j \in F'$ in this ordering with $j < i$.  If $E_i$ and $E_j$ are from the same row or the same
column, then since $F$ is $(\mathcal{I},\mathcal{J})$-adapted the condition on $E_i \cap E_j$ is
satisfied.  If $E_i$ and $E_j$ are in different rows or columns, the size of their intersection is
at most one.  If $E_i \cap E_j = \emptyset$ then the condition is trivially satisfied.  If $E_i \cap
E_j = \{u\}$, then $E_i$ must be from a row since $i > j$. Then $E_i$ does not contain any
$s_1,\dots,s_m$, so we must show that there is some $I \in \mathcal{I}$ so that $\phi_i(u) \in
I$.  This is true because $\mathcal{I}$ is full.  Thus $F'$ is $(\mathcal{I},\mathcal{J})$-adapted
at $s_1,\dots,s_m$.

Now apply Lemma~\ref{lem:embeddingatvertex} to $F'$ with  $m = f$, $s_1 = x_{0,0},\dots,s_f = x_{0,f-1}$,
$V_{m+1} = \cdots = V_{f^2} = V(H) - B$, and  $\gamma = \frac{1}{2}\alpha^{\sum d(x_{0,j})}
p^{|F|-\sum d(x_{0,j})}$. Ensure that $n_0$ is large enough and $\mu$ is small enough apply
Lemma~\ref{lem:embeddingatvertex} to show that
\begin{align*}
  \inj[F' \rightarrow H; x_{0,0} \rightarrow b_0, \dots, x_{0,f-1} \rightarrow b_{f-1}] \geq \gamma
  \left(\frac{n}{2}\right)^{f^2-f} = \frac{\gamma}{{2}^{f^2-f}} n^a.
\end{align*}
Each labeled copy of $F'$ produces a labeled $F$-absorbing set for $B$, so there are at least
$\frac{\gamma}{a!2^{f^2-f}} n^a$ $F$-absorbing sets for $B$.  The proof is complete by letting $\epsilon =
\frac{\gamma}{a!2^{f^2-f}}$.
\end{proof} 

\subsection{Almost perfect packings} 
\label{sub:almost}

In this section we prove that the conditions in Theorem~\ref{thm:cdpack} imply that there exists a
perfect $F$-packing covering almost all the vertices of $H$.

\begin{lemma} \label{lem:greedypacking}
  Let $k \geq 2$ and $\mathcal{I} \subseteq 2^{[k]}$ be a full antichain.  Fix $0 < p < 1$ and an
  $\mathcal{I}$-adapted $k$-graph $F$ with $f$ vertices.  Fix an integer $b$ with $f|b$.  For any
  $0 < \omega < 1$, there exists $n_0$ and $\mu > 0$ such that the following holds.  Let $H$ be an
  $k$-graph satisfying \discIk with $n \geq n_0$ and $f|n$.  Then there exists $C \subseteq V(H)$
  such that $|C| \leq \omega n$, $b| |C|$, and $H[\bar{C}]$ has a perfect $F$-packing.
\end{lemma}

\begin{proof} 
First, select $n_0$ large enough and $\mu$ small enough so that any vertex set $C$ of size
$\left\lceil \frac{\omega}{2} \right\rceil$ contains a copy of $F$.  To see this, let $\gamma =
\frac{1}{2}p^{|F|}(\frac{\omega}{2})^f$ and select $n_0$ and $\mu > 0$ according to
Lemma~\ref{lem:embeddingatvertex} with $m = 0$.  (Recall that if $m = 0$ then the condition
$(\mathcal{I},\mathcal{J})$-adapted on $F$ at $\emptyset$ just reduces to the statement that $F$ is
$\mathcal{I}$-adapted.) Now if $C \subseteq V(H)$ with $|C| \geq \frac{\omega }{2} n$, then let $V_1
=  \dots = V_f = C$ so that $|V_i| \geq \frac{\omega}{2}$ for all $i$.  Then
Lemma~\ref{lem:embeddingatvertex} implies there are at least $p^{|F|}\prod|V_i| - \gamma n^f \geq
p^{|F|} \left( \frac{\omega}{2} \right)^f n^f - \gamma n^f = \gamma n^f > 0$ copies of $F$ inside
$C$.

Now let $F_1, \dots, F_t$ be a greedily constructed $F$-packing.  That is, $F_1, \dots, F_t$ are
disjoint copies of $F$ and $C := V(H) - V(F_1) - \dots - V(F_t)$ has no copy
of $F$.  By the previous paragraph, $|C| \leq \frac{\omega}{2} n$.  Since $f|n$ and $H[\bar{C}]$ has
a perfect $F$-packing, $f| |C|$.  Thus we can let $y \equiv - \frac{|C|}{f} \pmod {b}$ with $0\leq y
< b$ and take $y$ of the copies of $F$ in the $F$-packing of $H[\bar{C}]$ and add their vertices
into $C$ so that $b| |C|$.
\end{proof} 

\subsection{Proof of Theorem~\ref{thm:cdpack}} 
\label{sub:cdproof}

\begin{proof}[Proof of Theorem~\ref{thm:cdpack}] 
First, apply Lemma~\ref{lem:richness} to produce $\epsilon > 0$.  Next, select $\omega > 0$
according to Lemma~\ref{lem:absorbing} and $\mu_1 > 0$ according to Lemma~\ref{lem:greedypacking}.
Also, make $n_0$ large enough so that both Lemma~\ref{lem:absorbing} and~\ref{lem:greedypacking} can
be applied.  Let $\mu = \mu_1 \omega^{k}$.  All the parameters have now been chosen.

By Lemmas~\ref{lem:absorbing} and~\ref{lem:richness}, there exists a set $A \subseteq V(H)$ such
that $A$ $F$-absorbs $C$ for all $C \subseteq V(H) \setminus A$ with $|C| \leq \omega n$ and $b \mid
|C|$.  If $|A| \geq (1-\omega)n$, then $A$ $F$-absorbs $V(H)\setminus A$ so that $H$ has a perfect
$F$-packing.  Thus $|A| \leq (1-\omega)n$.  Next, let $H' := H[\bar{A}]$ and notice that $H'$
satisfies \texttt{Disc}$^{(k)}(\mathcal{I},\mathsmaller{\mathsmaller{\geq}}p,\mu_1)$ since $v(H')
\geq \omega n$ and
\begin{align*}
  \mu n^k \leq \frac{\mu}{\omega^k} v(H')^k = \mu_1 v(H')^k.
\end{align*}
Therefore, by Lemma~\ref{lem:greedypacking}, there exists a vertex set $C \subseteq V(H') = V(H)
\setminus A$ such that $|C| \leq \omega n$, $|C|$ is a multiple of $b$, and $H'[\bar{C}]$ has a
perfect $F$-packing.  Now Lemma~\ref{lem:absorbing} implies that $A$ $F$-absorbs $C$.  The perfect
$F$-packing of $A \cup C$ and the perfect $F$-packing of $H'[\bar{C}]$ produces a perfect
$F$-packing of $H$.
\end{proof} 

\section{Constructions} 
\label{sec:Constructions}

In this section, we prove Propositions~\ref{prop:constr-fail-kminusone}
and~\ref{prop:constr-fail-links} using the following construction.

\begin{constr}
  Let $k \geq 2$.  Let $A_n^{(k)}$ be the following probability distribution over $n$-vertex
  $k$-graphs.  Let $f : \binom{V(A_n^{(k)})}{k-1} \rightarrow \{0,\dots,k-1\}$ be a random $k$-coloring of
  the $(k-1)$-sets.  Make $E \in \binom{V(A_n^{(k)})}{k}$ an edge of $A_n^{(k)}$ if
  \begin{align*}
    \sum_{\substack{T \subseteq E \\ |T| = k-1}} f(T) \neq 0 \pmod k.
  \end{align*}
\end{constr}

\begin{lemma}\label{lem:constredges}
  Let $p = \frac{k-1}{k}$ and $\epsilon > 0$.  Then with probability going to one as $n$ goes to
  infinity,
  \begin{align*}
    \left| |A_n^{(k)}| - p \binom{n}{k} \right| < \epsilon n^k.
  \end{align*}
\end{lemma}

\begin{proof} 
Each $k$-set is an edge with probability exactly $p$, so $\mathbb{E}[|A_n^{(k)}|] = p\binom{n}{k}$.
A simple second moment argument then shows that with high probability the number of
edges is concentrated around $p \binom{n}{k}$.
\end{proof} 

\begin{lemma} \label{lem:constr-fails-cd}
  There exists a $\mu_0$ such that for all $0 < \mu < \mu_0$, with probability going to one as $n$
  goes to infinity, $A^{(k)}_n$ fails \cdk.
\end{lemma}

\begin{proof} 
Let $Z$ be the $(k-1)$-graph whose edges are all the $(k-1)$-sets colored zero.  Let $\Lambda = (Z,
\dots, Z)$ be the $\binom{[k]}{k-1}$-layout consisting of $Z$ in every coordinate.  Now any
$k$-clique $(z_1,\dots,z_k)$ of $\Lambda$ is not a hyperedge of $A_n^{(k)}$, since every $(k-1)$-subset of
$\{z_1,\dots,z_k\}$ has color zero.  This $\Lambda$ will show that $A_n^{(k)}$ fails \cdk if
$|K_k(\Lambda)|$ is large enough.  Each $k$-tuple of vertices is a $k$-clique with probability
$(\frac{1}{k})^k$, so $\mathbb{E}[|K_k(\Lambda)|] = k^{-k} (n)_k$.  A
simple second moment computation shows that $|K_k(\Lambda)|$ is concentrated around its expectation,
so with high probability for large $n$ we have that $|K_k(\Lambda)| \geq \frac{1}{10} k^{-k} n^k$.
Thus if $\mu_0 = \frac{1}{20} \frac{k-1}{k^{k+1}}$, we have that
\begin{align*}
  0 = |H \cap K_k(\Lambda)| < \frac{k-1}{k} |K_k(\Lambda)| - \mu n^k.
\end{align*}
\end{proof} 

\begin{lemma} \label{lem:constr-has-no-clique}
  Let $r = r_{k-1}(K^{(k-1)}_k, \dots, K^{(k-1)}_k)$ be the $k$-color Ramsey number, where the
  $(k-1)$-sets are colored and a monochromatic $k$-clique is forced.  Then $A_{n}^{(k)}$ has no copy of
  $K^{(k)}_r$.
\end{lemma}

\begin{proof} 
Let $X \subseteq V(A_n^{(k)})$ be such that $|X| = r$ and $A_n^{(k)}[X]$ is a clique.  Then by the property of
$r$, there exists a $Y \subseteq X$ such that $|Y| = k$ and all $(k-1)$-subsets of $Y$ have the same
color $c$.  But now
\begin{align*}
  \sum_{\substack{T \subseteq Y \\ |T| = k-1}} f(T) = ck = 0 \pmod k.
\end{align*}
Thus $Y \notin A_n^{(k)}$, which contradicts that $A_n^{(k)}[X]$ is a clique.
\end{proof} 

To show that $A_n^{(k)}$ satisfies \discIk when $\mathcal{I} \neq \binom{[k]}{k-1}$, we will use a theorem
of Towsner~\cite{hqsi-towsner14} that equates $\mathcal{I}$-discrepency with counting
$\mathcal{I}$-adapted hypergraphs.  Therefore, we prove that the count of any $\mathcal{I}$-adapted
hypergraph $F$ in $A_n^{(k)}$ is correct with high probability.

\begin{lemma} \label{lem:constr-sat-count}
  Let $p = \frac{k-1}{k}$ and let $\mathcal{I} \subseteq 2^{[k]}$ be an antichain such that
  $\mathcal{I} \neq \binom{[k]}{k-1}$.  Let $F$ be an $\mathcal{I}$-adapted $k$-graph.  For every
  $\mu > 0$, with probability going to one as $n$ goes to infinity, the number of labeled copies of
  $F$ in $A^{(k)}_n$ satisfies
  \begin{align*}
    \left| \inj[F \rightarrow A^{(k)}_n] - p^{|F|}n^{v(F)} \right| < \mu n^{v(F)}.
  \end{align*}
\end{lemma}

\begin{proof} 
Let $E_1, \dots, E_m$ be the ordering of edges in the definition of $F$ being $\mathcal{I}$-adapted.
First we shows that if $Q : V(F) \rightarrow V(A_n^{(k)})$ is any injection, then the probability that
$Q(E_i) \in A_n^{(k)}$ is exactly $p$ independently of if the edges $E_j$ with $j < i$ map to hyperedges
or not.  Indeed, since $\mathcal{I} \neq \binom{[k]}{k-1}$, let $I \in \binom{[k]}{k-1} -
\mathcal{I}$.  Now consider some $E_i$ and let $\phi_i : E_i \rightarrow [k]$ be the bijection from
the definition of $F$ being $\mathcal{I}$-adapted.  Now since $I \notin \mathcal{I}$, there is no $j
< i$ such that $\phi_i(E_i \cap E_j) = I$.  Thus conditioning on if the edges $E_j$ with $j < i$ map
to edges of $A_n^{(k)}$ or not potentially fixes the colors on $(k-1)$-subsets of $Q(E_i)$ besides the
$(k-1)$-subset indexed by $I$.  Since the color of $\{Q(x) : x \in E_i, \phi_i(x) \in I \}$ (which
has size $k-1$) has probability exactly $p$ to make the color sum of $Q(E_i)$ once all other colors
are fixed, with probability $p$ we have that $Q(E_i)$ is an edge.

Therefore, the probability that $Q$ is an edge-preserving map is $p^{|F|}$.  This implies that the
expected number of labeled copies of $F$ in $A_n^{(k)}$ is $p^{|F|} n (n-1) \cdots (n-v(F) + 1)$.  A
simple second moment calculation shows that with high probability the number of labeled copies of
$F$ in $A_n^{(k)}$ is $p^{|F|} n^{v(F)} \pm \mu n^{v(F)}$ for large $n$.
\end{proof} 

Lastly, we need to show that $A_n^{(k)}$ satisfies \discImone in every link for every $\mathcal{J}$.  We
could do that similar to the previous lemma by showing that the count of $\mathcal{J}$-adapted
$k$-graphs is correct, but instead are able to directly show that \discImone holds.

\begin{lemma} \label{lem:constr-sat-J-in-link}
  Let $\mathcal{J} \subseteq 2^{[k-1]}$ be an antichain and $\alpha = \frac{k-1}{k}$.  Then for
  every $\mu > 0$, with probability going to one as $n$ goes to infinity, $L(x)$ satisfies
  \discImone for each $x \in V(A_n^{(k)})$.
\end{lemma}

\begin{proof} 
Fix $x \in V(A_n^{(k)})$ and view $L_{A_n^{(k)}}(x)$ as a probability distribution over $(k-1)$-graphs with
vertex set $V(A_n^{(k)}) - x$.  That is, an element from this probability distribution is generated by
first generating $A_n^{(k)}$ and then outputting the link of $x$.  We claim that the probability
distribution $L(x)$ is isomorphic to the probability distribution
$G^{(k-1)}(n-1,\alpha)$.  To see this, consider $S \in \binom{V(A_n^{(k)}) - x}{k-1}$.  Then $S \in L(x)$
if
\begin{align*}
  \sum_{\substack{T \subseteq S \cup \{x\} \\ |T| = k-1}} f(T) \neq 0 \pmod k.
\end{align*}
We could rewrite this as
\begin{align*}
  f(S) \neq \sum_{\substack{T \subseteq S \\ |T| = k-2}} f(T \cup x) \pmod k.
\end{align*}
The sum on the left hand side is some integer $w_S$ between $0$ and $k-1$, so that $S$ is a
hyperedge of $L(x)$ if and only if the color of $S$ is not $w_S$.  Since this is for every $S$ and
the colors assigned to $S$ are mutually independent, $L(x)$ is isomorphic to
$G^{(k-1)}(n-1,\alpha)$.

The proof is now complete, since for large $n$ $G^{(k-1)}(n-1,\alpha)$ satisfies \discImone with
very high probability as follows.  Fix any $\mathcal{J}$-layout $\Lambda$.  Each $(k-1)$-clique in
$\Lambda$ is a hyperedge with probability $\alpha$ and two $(k-1)$-cliques are independent unless
one is a permutation of the other.  So divide $K_{k-1}(\Lambda)$ up into at most $(k-1)!$ sets $R_1,
\dots, R_{(k-1)!}$ such that within a single $R_i$ there are no $(k-1)$-tuples which are
permutations of each other.  Then the expected size of $H \cap R_i$ is $\alpha |R_i|$ and by
Chernoff's inequality,
\begin{align*}
  \mathbb{P}\Big[ \Big| |H \cap R_i| - \alpha |R_i| \Big| > \epsilon n^{k-1} \Big] < 2
  e^{-\epsilon^2 n^{2k-2}/2|R_i|}.
\end{align*}
Since $|R_i| \leq n^{k-1}$, the probability is at most $e^{-c n^{k-1}}$ for some constant $c$.
There are $(k-1)!$ sets $R_i$ and there are at most $2^{k-2} 2^{n^{k-2}}$ $\mathcal{J}$-layouts $\Lambda$, so
with probability at most $e^{-c' n^{k-1}}$, the link of $x$ fails \discImone.  There are $n$
vertices of $A_n^{(k)}$, so with probability at most $n e^{-c' n^{k-1}} \rightarrow 0$, there is some vertex $x$
of $A_n^{(k)}$ whose link fails \discImone.
\end{proof} 

\begin{proof}[Proof of Proposition~\ref{prop:constr-fail-kminusone}] 
As mentioned previously, to show that $A_n^{(k)}$ satisfies \discIk, we combine
Lemma~\ref{lem:constr-sat-count} with a theorem of Towsner~\cite{hqsi-towsner14} which is stated in
the language of $k$-graph sequences.  Converting from the probability distribution $A_n^{(k)}$ to a
$k$-graph sequence is very similar to the proofs of~\cite[Lemmas 30 and 31]{hqsi-lenz-poset12} so we
only briefly sketch the technique here.  By the previous lemmas and the probabilistic method, for
every $\mu > 0$ there exists an $n_0$ such that for every $n \geq n_0$ there exists some $k$-graph
satisfying the properties in the previous lemmas (has the right edge density, fails \cdk, no copy of
$K_r$, has the right count of all $\mathcal{I}$-adapted hypergraphs, and satisfies \discImone in the
links).  Construct a $k$-graph sequence $\mathcal{H} = \{H_n\}_{n \in \mathbb{N}}$ by diagonalization
by setting $\mu = \frac{1}{n}$.  

By Lemma~\ref{lem:constr-sat-count}, $\mathcal{H}$ satisfies the property that for every
$\mathcal{I}$-adapted $F$, $\lim_{n \rightarrow \infty} t_F(H_n) = p^{|F|}$ so by
\cite[Theorem 1.1]{hqsi-towsner14} $\mathcal{H}$ is \texttt{Disc}$_p[\mathcal{I}]$ (where $t_F(H_n)$
and \texttt{Disc}$_p[\mathcal{I}]$ are defined in \cite{hqsi-towsner14}).  Thus for large $n$, the
$k$-graphs in the sequence $\mathcal{H}$ are the $k$-graphs which prove
Proposition~\ref{prop:constr-fail-kminusone}.
\end{proof} 

\begin{proof}[Proof of Proposition~\ref{prop:constr-fail-links}] 
Let $G = G^{(k)}(n,p)$ be the random $k$-graph with density $p$.  Modify $G$ by picking a single
vertex $x \in V(G)$, removing all edges which contain $x$, and adding edges so that $L(x) =
A^{(k-1)}_n$.  Now the link of $x$ has no copy of $K^{(k-1)}_r$ so that $G$ has no perfect
$K^{(k)}_{r+1}$-packing.  Also, $G$ satisfies \cdk since the random $k$-graph satisfies \cdk (see
the proof of Lemma~\ref{lem:constr-sat-J-in-link}) and we only modified at most $n^{k-1}$
hyperedges.  By the previous lemmas, the link of $x$ fails \cdkmone and satisfies \discImone for all
$\mathcal{J} \neq \binom{[k-1]}{k-2}$.
\end{proof} 

\textit{Acknowledgments:} The authors would like to thank Daniela K\"uhn for suggesting the
relationship of this work to the Hypergraph Blowup Lemma.

\bibliographystyle{abbrv}
\bibliography{refs.bib}

\end{document}